\documentclass[12pt,leqno]{amsart}
\usepackage{amsmath,amssymb,amsfonts}
\usepackage{eucal}
\usepackage{enumerate}
\usepackage{graphicx}
\usepackage{tikz}
\newcommand \comment[1]{}           
\renewcommand \comment[1]{\emph{[#1]}}      

\newtheorem{lemma}{Lemma}

\newtheorem{proposition}[lemma]{Proposition}
\newtheorem{theorem}[lemma]{Theorem}
\newtheorem{problem}[lemma]{Problem}

\theoremstyle{definition}
\newtheorem{remark}[lemma]{Remark}
\newtheorem{example}[lemma]{Example}
\newtheorem{definition}[lemma]{Definition}

\renewcommand{\phi}{\varphi}                 
\renewcommand{\epsilon}{\varepsilon}

\newcommand\bbN{\mathbb{N}}

\newcommand\bbR{\mathbb{R}}
\newcommand\bbZ{\mathbb{Z}}

\newcommand\cA{{\mathcal A}}
\newcommand\cN{{\mathcal N}}

\newcommand\G{\Gamma}

\begin{document}

\begin{center}

\Large
{Linial arrangements and local binary search trees}

\vskip10pt
\normalsize

\vskip20pt

{David Forge \footnote{The research of the   author is supported by the TEOMATRO project, grant number ANR-10-BLAN 0207.}\\
Laboratoire de recherche en informatique UMR 8623\\
B\^at.\ 650, Universit\'e Paris-Sud\\
91405 Orsay Cedex, France\\[5pt]

E-mail: {\tt forge@lri.fr}}\\[10pt]

\end{center}

\small
 {\sc Abstract.}
We study the set of NBC sets (no broken circuit sets) of the \emph {Linial} arrangement and deduce
a constructive bijection to the set of local binary search trees. We then generalize this construction to two families of Linial type arrangements for which the bijections are with some $k$-ary labelled trees that we introduce for this purpose.



 \normalsize\\


\vskip 20pt

\section{Introduction}\label{intro}

An \emph{integral gain graph} is a graph whose edges are labelled
invertibly by integers; that is, reversing the direction of an edge
negates the label (the \emph{gain} of the edge).  The \emph{gainic
hyperplane arrangement}, $\cA[\Phi]$, that corresponds to an integral gain graph $\Phi$
is the set of all hyperplanes in $\bbR^n$ of the form $x_j-x_i=g$ for
edges $(i,j)$ with $i<j$ and gain $g$ in $\Phi$.  (See \cite[Section IV.4.1, pp.\ 270--271]{BG} or \cite{SOA}.) The use of graph representing "graphic" arrangements is common and we think that it should also be the case for gain graphs representing gainic arrangements. 

 In last ten years there has been much interest in real hyperplane
arrangements of this type, such as the Shi arrangement, the Linial
arrangement, and the composed-partition or Catalan arrangement. For
all these families,  the characteristic polynomials and the number of regions  have been found. For the Shi arrangement, Athanasiadis \cite{A} gave
a bijection between the regions and the parking functions.

In \cite{forge}, we  started to replace the study of the regions of such arrangements by  the study of the NBC sets of the corresponding gain graphs. This study works specially well in the case of the complete gain graphs with gains
in intervals $[a,b]$ with $a+b=0$ or 1. This permits us to give a  bijection between the  NBC sets of the braid arrangement and the increasing labelled trees and another one  between  the NBC sets of the Shi arrangements and the  labelled trees. For the other values such that $a+b=0$ or 1 we introduced $[a,b]$-trees to get similar bijections. 

In this paper, we do the same thing for the cases where $a+b=2$. The first case, where $a=b=1$, corresponds to the so called Linial arrangement. The construction for the Linial case goes to the \emph{local binary search trees} ($LBS$ for short) as proposed by Stanley et al.
For the  intervals of the form $[1,k+1]$ and $[-k+1,k+1]$ we introduce two $k$- generalizations of the LBS which are ($k$+1)-ary.


\section{Basic definitions}\label{defs}

An \emph{integral gain graph} $\Phi = (\G,\phi)$ consists of a graph
$\G=(V,E)$ and an orientable function $\phi: E \to \bbZ$, called the
\emph{gain mapping}.  Orientability means that, if $(i,j)$ denotes an edge
oriented in one direction and $(j,i)$ the same edge with the opposite
orientation, then $\phi(j,i) = -\phi(i,j)$.  We have no loops
but multiple edges are permitted. For the rest of the paper, 
we denote the vertex set by $V = \{1,2,\ldots,n\} =: [n]$ with $n\ge1$.
 We use the notations $(i,j)$ for an edge with endpoints $i$ and $j$, oriented from $i$ to $j$, and $g(i,j)$ for
such an edge with gain $g$; that is, $\phi(g(i,j))=g$.  (Thus $g(i,j)$ is the same edge as $(-g)(j,i)$. ) 
 A \emph{circle} is a connected 2-regular subgraph, or its edge set.  
 Writing a circle $C$ as a word $e_1e_2\cdots e_l$, the gain of $C$ is
$\phi(C):=\phi(e_1)+\phi(e_2)+\cdots+\phi(e_l)$; then it is well defined whether the gain is zero or non zero.
 A subgraph is called \emph{balanced} if every circle in it
has gain zero. We will consider most especially balanced circles.

Given a linear order $<_O$ on the set of edges $E$, a \emph{broken circuit} is the set
of edges obtained by deleting the smallest element in a balanced circle.  
A set of edges, $N\subseteq E$, is a \emph{no-broken-circuit set} (NBC set for short)
if it contains no broken circuit. This notion from matroid theory
(see \cite{Bjorner} for reference) is very important here. We denote by 
$\mathcal N$ the set of NBC sets of the gain graph. It is well known 
that this set depends on the choice of the order, but its cardinality does not.

We can now transpose some ideas or problems from hyperplane arrangements
to gain graphs. For any integers $a,b, n$, let $K_n^{ab}$ be the gain graph
built on vertices $V=[n]$ by putting on every edge $(i,j)$ all the gains
$k$, for $a\le k\le b$. These gainic arrangements are called sometimes  deformations of the braid 
arrangement or truncated arrangements.  We have four main examples coming from hyperplane arrangements.  
We denote by $B_n$ the gain graph $K_n^{00}$  and call it the \emph{braid gain graph}, by
$L_n$ the gain graph $K_n^{11}$  and call it the 
\emph{Linial gain graph},  by $S_n$ the gain graph $K_n^{01}$  and call it the \emph{Shi gain graph}
and finally by $C_n$ the gain graph $K_n^{-11}$  and call it the \emph{Catalan gain graph}.

\section{Height function of a balanced gain graph}

We introduce the notion of height function on an integral gain graph on the vertex set $[n]$.
A height function $h$ defines two important things for the rest of the paper: the induced 
gain graph $\Phi[h]$ of a gain graph $\Phi$ and an order $O_h$ on the set
of vertices extended lexicographically to the set of edges.

\begin{definition} For $V$ a finite set of  $\bbN$, a \emph{height function} on $V$   is a function $h$ from
$V$ to $\bbN$ such that $h^{-1}(0)\not=\emptyset$. 
The \emph{corner} of $V$ defined by the
height function is the smallest element of greatest height. 

\end{definition}


\begin{definition}
Let $\Phi$ be a connected and balanced integral gain graph on a set $V$ of integers. 
The height function of the gain graph is the unique function
$h_\Phi$ such that for every edge $g(i,j)$ we have $h_\Phi(j)-h_\Phi(i)=g$.
(Such a function exists iff $\Phi$ is balanced.)
\end{definition}

\begin{definition}
Let  $\Phi$ be a gain graph also on $V=[n]$ and let $h$ be a height function on  $V$. 
We say that an edge $g(i,j)$ is \emph{coherent with $h$} if $h(j)-h(i)=g$.
The subgraph $\Phi[h]$ of $\Phi$ \emph{selected by $h$}
is the gain graph on the same vertex set $V$ whose edges are the edges of $\Phi$ coherent with $h$.
\end{definition}

\begin{definition}
Given a height function $h$ on the set $V$, the order $O_h$ on the set $V=[n]$ is defined by $i<_{O_h}j$ iff $h(i)>h(j)$ or  ($h(i)=h(j)$ and $i<j$). 
The order $O_h$ is extended lexicographically to an order $O_h$ on the set of edges coherent with the height function.
\end{definition}

With this definition, the corner is the smallest vertex of the whole graph. The \emph{subcorner} of $V$ defined by the
height function is the smallest neighbour of the corner in $\Phi[h]$.  When $|V|=1$, the subcorner is by definition the corner. In the Linial case, where all gains are equal to 1, the subcorner  is the smallest vertex on the second level. When 0 is a gain value, the subcorner is the second vertex on the first level when it exists. 

To explain a little all these definitions, a height function is just putting
the vertices on different levels. An integral gain graph $\Phi$ on integers
 defines a height function $h_\Phi$ when it is balanced (a tree works).
On the other hand, given a height function $h$ and a gain graph $\Phi$,
we get a subgraph $\Phi[h]$ of $\Phi$ by keeping only the edges coherent with $h$.
Finally, the height function defines the order $O_h$ on the edges of the gain graph
which we will need to define the NBC sets.

The height function will play a very central role in all the paper. The constructions
for the Linial arrangement as well as for its generalisations will always work once the height function has been chosen. This property will have for consequence, once we also define a height function on the corresponding trees (the other side of the bijection), to have a finer correspondence such as fixing the root of the tree. 

%

\section{NBC sets and NBC trees in  gain graphs}\label{nbc-semi}

An NBC set in a gain graph $\Phi$ is basically an edge set, as it arises from matroid theory.  We usually assume an NBC set is a spanning subgraph, i.e., it contains all vertices.  Thus, an NBC tree is a spanning tree of $\Phi$.  Sometimes we wish to have non-spanning NBC sets, such as the components of an NBC forest; then we write of NBC \emph{subtrees}, which need not be spanning trees.

Given a height function $h$, a gain graph $\Phi$ and the linear order $<_{O_h}$ on the edges, they determine the set of NBC sets of the subgraph $\Phi[h]$ relative to the order $<_{O_h}$, 
denoted by ${\mathcal N}_O(\Phi[h])$. As always, this set depends on the choice of the order but  its cardinality does not.

\begin{lemma}
Given an NBC tree $A$ of height function $h$ with corner $c$, the forest $A\setminus c$ 
is a disjoint union of NBC subtrees 
of height functions $h_1$,...,$h_k$, and the orders $O_{h_i}$ are restrictions of the order $O_h$.\qed
\end{lemma}

It is known from matroid theory that the NBC sets of the semimatroid of an affine arrangement $\cA$, 
with respect to a given ordering $<_O$ of the edges, correspond to the regions of the arrangement \cite[Section 9]{PS}.  
The semimatroid of $\cA[\Phi]$ is the frame (previously ``bias'') semimatroid of $\Phi$, which consists of the balanced edge sets of the gain graph 
$\Phi$ (\cite[Sect.\ II.2]{BG} or \cite{SOA}).  Thus, the NBC sets of that semimatroid are the spanning forests of $\Phi$.  Therefore $|\cN_O(\Phi)|$ equals the number of regions of $\cA[\Phi]$.

We show that the total number of NBC trees in an integral gain graph $\Phi$ equals the sum, over all height functions $h$, of the number of NBC trees in $\Phi[h]$.

Let $\Phi$ be connected.  Then we can decompose $\cN_O(\Phi)$ into disjoint subsets $\cN_O(\Phi[h])$, one for each height function $h$ that is coherent with $\Phi$ (that means that $\Phi[h]$ is also connected).  We have now: 
\[
\cN_O(\Phi) = \biguplus \{ \cN_O(\Phi[h]) \mid h \text{ is coherent with } \Phi \}.
\]

Therefore, 
the total number of NBC trees of all $\Phi[h]$ with respect to all possible height functions $h$ equals the number of NBC trees of $\Phi$.



\section{Complete $[a,b]$-gain graphs and their NBC trees}

Let $a$ and $b$ be two integers such that $a\le b$. 
The interval $[a,b]$ is the set $\{i\in \bbZ \mid a\le i\le b\}$.
We consider the gain graph $K^{ab}_n$ with vertices labelled by $[n]$
and with all the edges $g(i,j)$, with $i<j$ and $g\in [a,b]$. The arrangements that correspond to these
gain graphs, called deformations of the braid arrangement, have been of particular interest.  
The braid arrangement corresponds to the special case $a=b=0$. 
Other well studied cases are $a=-b$ (extended Catalan), $a=b=1$
(Linial) and $a=b-1=0$ (Shi).

We will describe the set of NBC trees of $K^{ab}_n[h]$ for a given
height function $h$.
The idea is that, as mentioned above, the height function $h$ defines
an order $O_h$ on a balanced subgraph. We will then be able to describe the NBC sets
coherent with $h$ for the order $O_h$.


\begin{proposition}
Let $a$ and $b$ be integers such that $a\le b$.  
Let $h$ be a height function of corner $c$ and let $\Phi$ be a spanning tree of 
$K^{ab}_n[h]$.  Suppose $c$ is incident to
the edges $g_i(c,v_i)$, $1\le i\le k$, and let $\Phi_i$ be the  connected
component of $\Phi\setminus c$ containing $c_i$ (that is a subtree). 
Then  $\Phi$   is an NBC tree if and only if
all the $\Phi_i $ are NBC trees and each $v_i$ is the $O_h$-smallest  vertex of $\Phi_i$ adjacent to $c$ in $K^{ab}_n[h]$.
\label{david}
\end{proposition}
\begin{proof}
Everything comes from the choice of the order $O_h$ for the
vertices and the edges. If we have a vertex $v$ in $\Phi_i$
such that $v<_{O_h}v_i$ for which the edge $(c,v)\in K^{ab}_n[h] $ exists then this edge 
is smaller than all the edges of $\Phi _i +c$. Such an edge then closes a balanced
circuit being the smallest edge of the circuit which is not possible.

In the other direction, if $\Phi$ is not an NBC tree then there is an 
edge $(x,y)$ in $ K^{ab}_n[h] $ closing a balanced circuit by being the
smallest edge of the circuit. Since the $\Phi_i$ are by hypothesis are NBC trees
the vertices $x$ and $y$ can not be in a same $\Phi_i$. They can not be
in two different $\Phi_i$ neither since the smallest edge would contain
$c$ necessarily. The last solution is that one of the vertices say $x$ is $c$
and that the other vertex $y$ is in a $\Phi_i$. Since the edge $(c,v_i)$
will be in the circuit we need to have  $(x,y)<_{O_h}(c,v_i)$.
This implies the condition of the proposition.
\end{proof} 

\section{Local binary search trees and two generalisations }

The  local binary search trees (LBS for short) are labelled rooted plane binary trees such that
a vertex has possibly two children a left one and a right one with
the property that : the value of the parent is bigger than the value of the left child 
and smaller than the value of the right child. The number of LBS labelled on the set $[n]$
is known to be equal to the number of regions of the Linial arrangement in dimension $n$.

\begin{definition}
A LBS is called {\it a left LBS } (LLBS for short) (resp. {\it a right LBS (RLBS for short)}) if the root has no right child (resp. left child).
\end{definition}

Let $T$ be a LBS tree of root $r=r_0$. Let $r_1$ be its  right child and $r_2$
be $r_1$'s right child and so on... $r_{k+1}$ be $r_k$'s right child. Let $r_\ell$ the last of these
vertices. If we delete the edges $\{r_i,r_{i+1}\}$ we obtain $\ell +1$ disjoint
 LLBS  $L_0$ , \ldots ,$L_\ell$. We call this the {\it left decomposition} of an LBS.

Similarly we have the right decomposition of an LBS by taking $r=r'_0$. Let $r'_1$ be its  left child and $r'_2$
be $r'_1$'s left child and so on... $r'_{k+1}$ be $r'_k$'s left child. Let $r'{_\ell '}$ the last of these
vertices. If we delete the edges $\{r'_i,r'_{i+1}\}$ we obtain $\ell ' +1$ disjoint
 RLBS $R_0$ , \ldots ,$R_{\ell'}$. We call this the {\it right decomposition} of an LBS..

We introduce now a generalization of LBS which will have many similar properties.
The  local $k$-ary search trees (L$k$S for short) are labelled rooted plane $k$-ary trees such that
a vertex has possibly $k$ children numbered from 1 to $k$ such  that  the value of the parent is bigger 
than the value of the number 1  child 
and smaller than the value of the number $k$ child. Note that a vertex has any number of children
from 0 to $k$ but that the number of the child does not depend on the presence of
the other children (this is what means plane $k$-ary). Note also that of course a LBS is a L$k$S for
$k=2$.

\begin{definition}
A L$k$S is called {\it a left L$k$S } (LL$k$S for short) 
 if the root has no number $k$ child.
\end{definition}

%

Let $T$ be a L$k$S tree of root $r=r_0$. Let $r_1$ be its  number $k$ child and $r_2$
be $r_1$'s number $k$ child and so on... $r_{k+1}$ be $r_k$'s number $k$ child. Let $r_\ell$ the last of these
vertices. If we delete the edges $\{r_i,r_{i+1}\}$ we obtain $\ell +1$ disjoint
 LL$k$S  $L_0$ , \ldots ,$L_\ell$. We call this the {\it left decomposition} of an L$k$S.
If we start with a LL$k$S tree $T$ (already left) of root and subtrees $T^i$ for $1\le i\le k-1$, then each $T^i$ has a left decomposition $D^i=\{L_0^i, L_1^i,\ldots,L^i_{\ell_i}\}$. We call the set of the $D^i$ the left decomposition of $T$.

%

We can define a height function of an LL$k$S of root $r$ and left decompositions
$D_i=\{L_0^i, L_1^i,\ldots,L^i_{\ell_i}\}$ recursively by: define $h(r)=0$ and take a height function
$h_j^i$ in each $L_j^i$. Then define $h$ on all the roots of the $L_j^i$ by:
\begin{itemize}
\item if $r_j^i$ is the root of a $L_j^i$ and is smaller than $r$
take $h(r_j^i)=-(k-i+1)$; 
\item if $r_j^i$ is the root of a $L_j^i$ is bigger than $r$
take $h(r_j^i)=-(k-i+1)+1$. 
\end{itemize}
We just need to take for the other vertices $h(x)=h_j^i(x) +h(r_j^i)-h_j^i(r_j^i)$ if  $x\in L_j^i$.

Note that $r$ is the corner of this height function. Note also that the vertices of weight 1 are in the lowest levels. This choice comes from the future bijection with the NBC sets. It could be easily reversed. Note also that the level of $r_j^i$ the root of $L_j^i$ has two the possible values $-k+i-1$ or 
$-k+i$ depending on the fact $r_j^i<r$ or not.

We introduce now another generalization of LBS corresponding to another generalization of the Linial arrangement.
The semi local $k$-ary search trees (SL$k$S for short) are labelled rooted plane $k$-ary trees such that
a vertex has possibly $k$ children numbered from 1 to $k$ such  that  the value of the parent is bigger 
than the value of the first  child from 1 to $k-1$
and smaller than the value of the number $k$ child.  Note that of course a LBS is a SL$k$S for $k=2$. Note also that a  SL$k$S tree is also a  L$k$S tree and that therefore we will not need to define the decomposition or  the height function of this new family.

\section{Linial gain graph}

In \cite{forge}, we have studied with more details the cases $a+b=0$ and $a+b=1$ which
contain the braid and the Shi cases. Now we will consider the case $a+b=2$ and starting by  $a=b=1$ corresponding to the Linial case.

The Linial gain graph is the first $K^{ab}_n$ graph with $a+b=2$ by taking $a=b=1$.  It corresponds to the
Linial arrangement, whose hyperplanes have equation $x_i-x_j=1$, with $i<j$. 
The number of regions (and then of NBC sets) of the Linial arrangement
in dimension $n$ is known to be equal to 
$$\frac{1}{n2^{n-1}}\sum_{k=1}^n{{n}\choose{k}} k^{n-1}.$$ It is also
known to be equal to the number of local binary search trees on $n$ vertices. 
We will give now a bijection between   the NBC sets of the Linial arrangement
on $[n]$ and the LBS on $[n]$. We start by giving a bijection between   
the NBC trees of the Linial arrangement on $[n]$ and the LLBS on $[n]$.

\begin{theorem}
The number of NBC of the Linial gain graph on $[n]$ is equal to the number of LBS on $[n]$.
Moreover, the number of NBC of corner $c$ of the Linial gain graph  on $[n]$ is equal to 
the number of LBS of root $c$ on $[n]$. And even for any height function on $[n]$, the number of NBC of height $h$ of the Linial gain graph  on $[n]$ is equal to the number of LBS of height $h$ on $[n]$.
\end{theorem}

\begin{example}
In the figure, we give on the left a LLBS and on the right the NBC tree corresponding. The special point in both sides is 3; in the left side if it the biggest number smaller than the root 4 in the right chain $1-3-5-7$ and in the right side it is the biggest neighbour of the corner. 

From left to right, the node which are smaller than the root becomes neighbours of the corner. The two other node 5 and 7 will be corners of their subtree but are connected differently. The subtree of root 5 gives a sub NBC tree of corner 5 (always the root becomes the corner) and of sub corner 2. Since 2 is smaller than the root 4, this subtree is connected by the edge $\{4,2\}$. The subtree of root 7 gives a sub NBC tree of corner 7  and of sub corner 6. Since 6 is bigger than the root 4, this subtree is connected by the edge $\{3,7\}$. 

In the other direction, 3 is again recognized as the special vertex and gives the decomposition in subtrees. The LLBS obtained are just making the chain of the left decomposition of the LLBS of root 4. 

\end{example}

\begin{tikzpicture}
\node [draw] (A) at (2,0) {4} ;
\node [draw] (B) at (0,-1) {1} ;
\node [draw] (C) at (2,-2) {3} ;
\node [draw] (D) at (4,-3) {5} ;
\node [draw] (E) at (2,-4) {2} ;
\node [draw] (F) at (6,-4) {7} ;
\node [draw] (G) at (4,-5) {6} ;
\draw (A) -- (B);
\draw (C) -- (B);
\draw (C) -- (D);
\draw (D) -- (E);
\draw (D) -- (F);
\draw (F) -- (G);

\node [draw] (A') at (10,0) {4} ;
\node [draw] (B') at (10,-2) {1} ;
\node [draw] (C') at (12,-2) {3} ;
\node [draw] (D') at (14,-0) {5} ;
\node [draw] (E') at (14,-2) {2} ;
\node [draw] (F') at (16,0) {7} ;
\node [draw] (G') at (16,-2) {6} ;
\draw (A') -- (B');
\draw (C') -- (A');
\draw (C') -- (F');
\draw (A') -- (E');
\draw (D') -- (E');
\draw (F') -- (G');

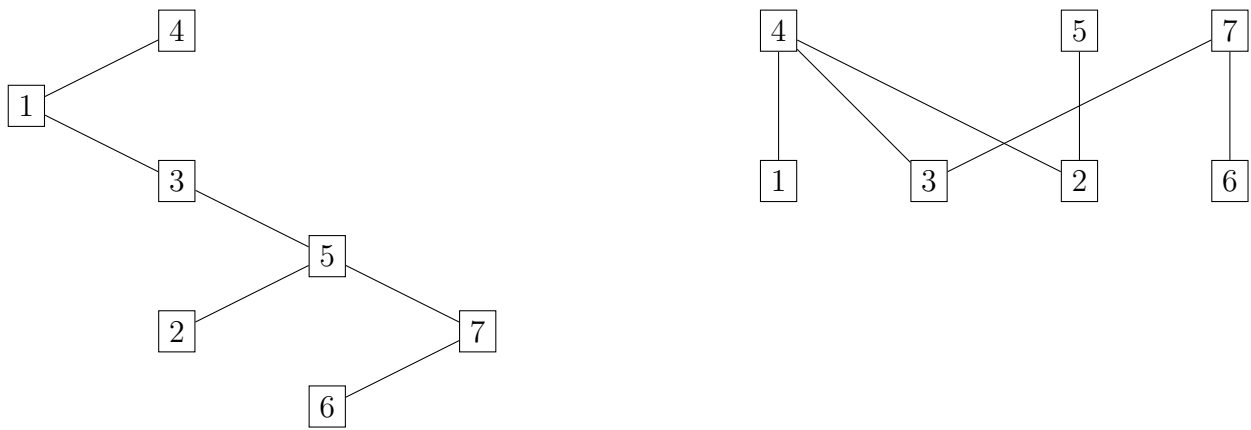
\node at (8,-7) {Figure 1 : Correspondance from LLBS to NBC};
\end{tikzpicture}

\begin{proof}
From the definitions of height functions, the following correspondence preserves height function in both directions and therefore corner goes to root. The fact that the height function is preserved is of course very important but it also forces the position of the vertices.

To go from LLBS to LBS we just use the left decomposition which correspond to a partition of $[n]$
in the same way as a NBC is just a union of NBC trees (we could call this a decomposition of an NBC in NBC trees!).

\medskip
{\bf(From LLBS to NBC trees)} Let $L$ be a LLBS with root $r$ and left decomposition 
in LLBS $L_i$ of root $r_i$ for $1\le i\le k$. We have $r_1<r_2<\cdots<r_k$ and let $r_p$
be the biggest $r_i$ smaller than $r$. The bijection is recursive as follows:
\begin{itemize}
\item Let $N_i$ be the NBC trees corresponding to $L_i$ for $1\le i \le k$. They are of corner $r_i$. 
\item Add the edge $ \{r, r_i\}$ for $1\le i\le p$.
\item For $p+1\le i\le k$, let $c_i$ and $sc_i$ be the corner and subcorner of $N_i$. By
hypothesis $c_i$ is bigger than $r$ and also $r_p$. If $sc_i<r_p$  then add the edge $ \{r, sc_i\}$
and otherwise add the edge $ \{r_p, c_i\}$. 
\end{itemize}

\medskip
{\bf(From NBC trees to LLBS)} Let $T$ be a NBC tree of height function $h_T$ and so for the order 
$O_{h_T}$. Let $c$ be its corner and $v_j^i$ its neighbours and $T_j^i$ the subtrees
obtained by deleting the vertex $c$ ( by notation the vertex $v_j^i$ uses gain $i$)  . Note that the vertex $v_j^i$ is not necessarily the corner of 
the subtree $T_j^i$ but that it can also be the subcorner. The vertex $p$ which is the smallest for $O_h$ neighbour of the corner plays a special role. The vertex $p$ is the biggest vertex using gain $k$.  We need to recognize the pieces which were attached to $N_p$ in the preceding
 construction. The vertex $p$ by the order and the fact that $T$ is an NBC tree has at most one neighbour  $v'_1$ bigger and $(p,v'_1)$ has gain 1. The vertex $v'_1$ is not necessarily the corner of $T\setminus N_p$ but can have a neighbour $v'_2$ with edge of gain 0 and $v'_2<v'_1$.  Let $T'_0,T'_1,\ldots,T'_\ell$ the subtrees of $T_p$
obtained by deleting the edges $\{v'_i,v'_{i+1}\}$. The subtree $T'_0$ is the subtree
containing $p$ and the subtree $T'_i$ is the subtree containing $v'_i$.  The bijection is reccursive:
\begin{itemize}
\item Let $L_i$ be the LLBS corresponding to $T_i$ for $1\le i \le k-1$ of corner $c_i$
\item Let $L'_j$ be the LLBS corresponding to $T'_i$ for $0\le i \le \ell$ of corner $c_i$
\item Let relabel the $L_i$ and $L'_i$ to some $L''_i$ along their roots $r_i$. That is 
we have $k+\ell$ LLBS $L''_i$ such that $r_1<r_2<\cdots<r_{k+\ell}$. 
\item Add the edge $ \{c,r_1\}$ (a left edge) and all the right edges $\{r_i,r_{i+1}\}$ for
$1\le i\le k+\ell-1$.
\end{itemize}

\end{proof}

We give now a more surprising correspondence between RLBS of root $r$ and NBC trees of subcorner $r$. Clearly we can go by symmetry from the RLBS to LLBS so the surprise is that that there is also a correspondence between LLBS of corner $c$ and of subcorner $n-c+1$.

\begin{theorem}
The number of NBC of the Linial gain graph on $[n]$ is equal to the number of LBS on $[n]$.
Moreover, the number of NBC of subcorner $c$ of the Linial gain graph  on $[n]$ is equal to 
the number of RBS of root $c$ on $[n]$. 
\end{theorem}

\begin{proof}

\medskip
{\bf(From NBC trees to RLBS)} Let $N$ be a NBC tree of height function $h_N$ and so for the order 
$O_{h_N}$. Let $sc$ be its subcorner and $v_1<v_2<\ldots<v_k$ its neighbours and $N_1,\ldots ,N_k$ the subtrees
obtained by deleting the vertex $sc$. Note that the vertex $v_i$ is whether  the corner $c_i$ 
or the subcorner $sc_i$ of the subtree $N_i$. From the fact that $sc$ is the subcorner
of $N$, we deduce that there is at least one of the $N_i$ of subcorner $sc_i$ bigger than $sc$.
 The bijection is recursive:
\begin{itemize}
\item Let $R_i$ be the RLBS corresponding to $N_i$ for $1\le i \le k$ of root $sc_i$.
\item Let relabel the $R_i$ in such a way  that $sc_1>sc_2>\cdots>sc_k$. 
\item Add the edge $ \{sc,sc_1\}$ (a right edge) and all the left edges $\{sc_i,sc_{i+1}\}$ for
$1\le i\le k-1$.
\end{itemize}

\medskip
{\bf(From RLBS to NBC trees)} Let $R$ be a RLBS with root $r$ and right decomposition 
in RLBS $R_i$ of root $r_i$ for $1\le i\le k$. We have $r_1>r_2>\cdots>r_k$ and let $r_p$
be the biggest $r_i$ smaller than $r$. The bijection is recursive as follows:
\begin{itemize}
\item Let $N_i$ be the NBC tree corresponding to $R_i$ for $1\le i \le k$. They are of subcorner $sc_i$
and of corner $c_i$ with $c_i\ge sc_i$ (recall that the subcorner in a one vertex tree is the corner). 
\item If $sc_i>sc$ then add the edge $ \{sc, c_i\}$ for $1\le i\le p$.
\item If $sc>c_i$ then add the edge $ \{sc, c_i\}$ for $1\le i\le p$.
\item If $c_i>sc>sc_i$ then add the edge $ \{sc, sc_i\}$ for $1\le i\le p$.
\end{itemize}

\end{proof}

Here is now a third decomposition of NBC trees by its corner but the pieces are 
given by their subcorner. 
There is a fourth decomposition of a RLBS tree into LLBS which is missing.

\medskip
{\bf(From NBC trees to LLBS)} Let $N$ be a NBC tree of height function $h_N$ and so for the order 
$O_{h_N}$. Let $c$ be its corner and $v_1<v_2<\ldots<v_k$ its neighbours and $N_1,\ldots ,N_k$ the subtrees
obtained by deleting the vertex $c$. Note that the vertex $v_i$ is whether  the corner $c_i$ 
or the subcorner $sc_i$ of the subtree $N_i$. From the fact that $c$ is the corner
of $N$, we deduce that all the $N_i$  have a subcorner $sc_i$ smaller than $c$.
 The bijection is recursive, we build $R$ of root $c$ by:
\begin{itemize}
\item Let $R_i$ be the RLBS corresponding to $N_i$ for $1\le i \le k$ of root $sc_i$.
\item Let relabel the $R_i$ in such a way  that $sc_1>sc_2>\cdots>sc_k$. 
\item Add the edge $ \{c,sc_1\}$ (a left edge) and all the left edges $\{sc_i,sc_{i+1}\}$ for
$1\le i\le k-1$.
\end{itemize}

\medskip
{\bf(From LLBS to NBC trees)} Let $L$ be a LLBS with root $r$ and right decomposition 
in RLBS $R_i$ of root $r_i$ for $1\le i\le k$. We have $r_1>r_2>\cdots>r_k$. 
The bijection is recursive, we build $N$ of corner $r$ by::
\begin{itemize}
\item Let $N_i$ be the NBC tree corresponding to $R_i$ for $1\le i \le k$. They are of subcorner $r_i$
and of corner $c_i$ with $c_i\ge r_i$ (recall that the subcorner in a one vertex tree is the corner). 
We have for all $N_i$ that the subcorner is smaller then $r$ the corner of the full NBC tree.
\item If $c>c_i$ then add the edge $ \{c, c_i\}$ for $1\le i\le k$.
\item If $c<c_i$ then add the edge $ \{c, sc_i\}$ for $1\le i\le k$.
\end{itemize}


\section{Other $[a,b]$  gain graph with $a+b=2$}

We are now considering the complete gain graph $K_n^{2-k,k} $ for any integer $k\ge 1$. 
The case $k=1$ is the preceding Linial case.

\begin{theorem}
The number of NBC of the $K_n^{2-k,k} $ gain graph on $[n]$ is equal to the number of L$k$S on $[n]$.
Moreover, the number of NBC of corner $c$ of the $K_n^{2-k,k} $ gain graph  on $[n]$ is equal to 
the number of L$k$S of root $c$ on $[n]$. And even for any height function $h$ on $[n]$, the number of NBC of height $h$ of the $K_n^{2-k,k} $ gain graph  on $[n]$ is equal to the number of L$k$S of height $h$ on $[n]$.
\end{theorem}

\begin{proof}
The following correspondence preserves height function in both directions and therefore corner goes to root. The fact that the height function is preserved is of course very important but it also forces the position of the vertices. When we give the gain of an edge there is then no more choice. In the other direction the weight of the edge is also without choice and comes directly from the height function. There is a special vertex $p$ which plays a very central role. In one direction it permits to connect all the pieces even those which can not be connected directly to the root. In the other direction, it permits to recognize all the pieces.

To go from LL$k$S to L$k$S we just use the left decomposition which correspond to a partition of $[n]$ is the same way as a NBC is just a union of NBC trees. So we just ned to go make the correspondence between LL$k$S and NBC trees.

\medskip
{\bf(From LL$k$S to NBC trees)} Let $L$ be a LL$k$S with root $r$ and left decomposition 
in LL$k$S $D^i=\{L_0^i, L_1^i,\ldots,L^i_{\ell_i}\}$ each $L_j^i$ being of root $r_j^i$. Let $p$
be the biggest $r_j^1$ smaller than $r$. The bijection is recursive as follows:
\begin{itemize}
\item Let $N_j^i$ be the NBC trees corresponding to $L_j^i$. They are of corner $r_j^i$. 
\item For $j>1$, every $N_j^i$ can be uniquely connected with an edge to $(r,r_j^i)$. If $r_j^i<r$ add the edge $(r,r_j^i)$ of gain $k-i+1$ and if $r_j^i>r$ add the edge $(r,r_j^i)$ of gain $i-k$.
\item For $j=1$ we are in a case similar to  the Linial case.  Let $sc_j^i$ be the subcorner of $N_j^i$.
\begin{itemize}
\item Add the edge $ \{r, r_j^i\}$ for $r_j^i \le p$ with gain $k$.
\item Add the edge $ \{r, r_j^i\}$ for $r_j^i >p$  but $sc_j^i<_{O_h}p$ with gain $k$. All the other $N_j^i$ will be connected to $p$ but not necessarily directly.
\item For $r_j^i> p$ and $sc_j^i>_{O_h}p$, by
hypothesis $c_j^i$ is bigger than $r$ and also than $p$. If $sc_j^i<p$  then add the edge $ (r, sc_j^i)$ with gain $k$. For the $N_j^i$ such that $sc_j^i>p$ taken in increasing order, we add the edge  $ (p, r_j^i)$ with gain 1 for the biggest $r_j^i$. Then we add  the edge $(r_{j'}^i,r_j^i)$ with gain 0 ($r_{j'}^i$ being the corner of the preceding such $N_j^i$ taking the $r_j^i $ in decreasing order). The reason of this construction apart from keeping the height function is to be able to recognize the pieces in the other direction construction.
\end{itemize}
\end{itemize}

\medskip
{\bf(From NBC trees to LL$k$S)} 
The case $k=1$ is Linial and there is nothing to prove. In the case $k\ge2$, we will use that  the value 0 is a possible gain. 
The construction is very similar to the Linial case. The construction is just done on the neighbours with gain $k$. The other vertices can be directly attached to the corner in any case.

 Let $N$ be a NBC tree of height function $h_N$ and so for the order 
$O_{h_N}$. Let $c$ be its corner and $v_j^i<c$ its neighbours and $N_j^i$ the subtrees
obtained by deleting the vertex $c$ and the corresponding edge of gain $i$ by notation. From Proposition \ref{david}, the vertex $v_j^i$ is not necessarily the corner 
of the subtree $N_j^i$.   In fact this is true only if the gain is $k$. For
all other gain, $v_j^i$ is necessarily the corner of $N_j^i$. The reason comes from
the choice of the order $O_h$. The vertex $p$ which is the biggest
neighbour of the corner with gain $k$ plays a special role and let $N_p$ be its component. 
From the choice of the order, the vertex $p$ has at most one neighbour $v'_1$ bigger using gain 1. This vertex $v'_1$ must verify $v'_1>r$. 
The vertex $v'_1$ can  have only one  neighbour $v'_2$ smaller using gain 0. Similarly $v'_2$ can also have only one smaller neighbour using gain 0. After deleting all the edges $(v'_i,v'_{i+1})$ we get NBC subtrees $N'_i$ whose corner is $v'_i$. 
 The subtree $N'_0$ is as well the subtree
containing $p$ and is of corner $p$.  The bijection is recursive :
\begin{itemize}
\item Let $L_j^i$ be the LL$k$S corresponding to $N_j^i$  of corner $v_j^i$. We have $h(c)-h({v}_j^i)=i$ and ${v}_j^i<c$.
\item Let ${L'}_i$ be the LL$k$S corresponding to ${N'}_i$ of corner ${v'}_i$. We have $h(c)-h({v'}_j^i)=k-1+i$ and ${v'}_j^i>c$.
\item To every $L_j^i$ such that $i>1$ and $v_j^i<c$ put weight $k-i+1$.
\item To every $L_j^i$ such that $i>1$ and $v_j^i>c$ put weight $k-i$.
\item To every $L_j^i$ such that $i=1$  put weight $1$.
\item To every $L'_i$ put weight $1$.
\item The LL$k$S tree is given by its root $c$ and itsleft decomposition $D_w$ where $D_w$ is just the set of sub LL$k$S trees to which we have put weight $w$. 
\end{itemize}

\end{proof}


\section{$[1,k]$  gain graph}

We are now considering the complete gain graph $K_n^{1,k} $ for $k$ any integer bigger than 1. 
The case $k=1$ is the preceding Linial case. The bijection is based on the same idea: to connect all the blocks to the root or a special vertex. Here it becomes a little more complicated because of the different possible  gains.

\begin{theorem}
The number of NBC of the Linial gain graph on $[n]$ is equal to the number of SL$k$S on $[n]$.
Moreover, the number of NBC of corner $c$ of the Linial gain graph  on $[n]$ is equal to 
the number of SL$k$S of root $c$ on $[n]$. And even for any height function on $[n]$, the number of NBC of height $h$ of the Linial gain graph  on $[n]$ is equal to the number of SL$k$S of height $h$ on $[n]$.
\end{theorem}

\begin{proof}
The following correspondence preserves height function in both directions and therefore corner goes to root.
To go from LLBS to LBS we just use the left decomposition which correspond to a partition of $[n]$ is the same way as a NBC is just a union of NBC trees. So we just need to  make the correspondence between SLL$k$S and NBC trees.

\medskip
{\bf(From LL$k$S to NBC trees)} Let $L$ be a LL$k$S with root $r$ and left decomposition 
in LL$k$S $D^i=\{L_0^i, L_1^i,\ldots,L^i_{\ell_i}\}$ each $L_j^i$ being of root $r_j^i$. Let $p$
be the biggest $r_j^1$ smaller than $r$.  The height function $h$ of $L$ will be needed. The bijection is recursive as follows:
\begin{itemize}
\item Let $N_j^i$ be the NBC trees corresponding to $L_j^i$. They are of corner $r_j^i$.  Let $x_j^i$ be, if it exists,  the smallest vertex (for order $O_h$) such that an edge $(r,x_j^i)$ exists in $\Phi[h]$. Previously, this vertex could be only the corner or the subcorner. Now it can be  on any level between 1 and  $h(r)-h(r_j^i)$. However it is the smallest element of its level.
\item We will connect the $N_j^i$ to $r$ by taking them in the decreasing order for $O_h$ of their corner $r_j^i$. The first such $N_i^j$ is then $N_p$ of corner $p$. Some of the $N_j^i$ will be connected directly to the corner $r$. The other $N_j^i$ will be connected to $N_p$ making a special component which will be growing during the construction: we call this component $T$ which is a name of variable. 
\item Add the edge $(r,p)$ of gain $k-\ell+1$. 
\begin{itemize}
\item If $x_j^i<_h p$ then add the edge the edge $(x_j^i,p)$. For th gain there are two cases : if
 $r_j^i<r$ then the gain is $k-i+1$ and else of gain $k-i+h(r_j^i)-h(x_j^i)$  (the $+1$ disappeared to conserve the height function).
This is the case where $N_j^i$ can be connected directly to the corner $r$. The other case will go to $T$.
\item If $x_j^i>_h p$ or also $x_j^i$ does not exist now we connect to $T$. We just need to find in $T$ the vertex $t_j^i$ to add edge $(t_j^i,r_j^i)$. The vertex $t_j^i$ is simply the biggest vertex
of $T$  for $O_h$ such that the edge $(t_j^i,r_j^i)$ exists. So we add the edge $(r_j^i,t_j^i)$ with gain $h(r_j^i)-h(t_j^i)$.

\end{itemize}
\end{itemize}

To be complete we must check that the sign graph is indeed a NBC tree of height function $h$ and such that $p$ is the biggest neighbour of $r$. First, the tree $T$ is at every moment a NBC tree by the choice of $t_j^i$. Also the condition $x_j^i>_h p$ makes that the vertex $p$ is always the biggest vertex of $T$ for $O_h$ which can be connected to $r$. So by Proposition \ref{david} we obtain an NBC tree. The vertex $p$ is then the biggest neighbour of $r$ since every other neighbour is a $x_j^i$ verifying  $x_j^i<_h p$. Finally the height function is well preserved by the choices of the gains.

\medskip
{\bf(From NBC trees to LL$k$S)} 
The case $k=1$ is Linial and there is nothing to prove. In the case $k\ge2$, we will use that  the value 0 is a possible gain. 

 Let $N$ be a NBC tree of height function $h_N$ and so for the order 
$O_{h_N}$. Let $c$ be its corner and $v_j^i<c$ its neighbours, $N_j^i$ the subtrees
obtained by deleting the vertex $c$, $c_j^i$ the corner of $N_j^i$. By notation we have that 
 $h(c)-h(c_j^i)=i$ by notation. If we want the gain of the edge$(c,v_j^i)$ is is $h(c)-h(v_j^i)$.
 From Proposition \ref{david}, the vertex $v_j^i$ is not necessarily the corner 
of the subtree $N_j^i$ but the smallest vertex for $O_h$ which can be connected in $N_j^i$.   The vertex $p$ which is the biggest 
neighbour of the corner for $O_h$  plays a special role and let $N_p$ be its component. 

We need now to cut $N_p$ to get back the blocks in the preceding construction ($N_p$ corresponds to the final $T$). 
The algorithm to get back the pieces is simply a kind of depth first search starting at $p$.
Each time we found an edge $(x,y)$ such that $y<_h p$ we know that the vertex $y$ is not in $T_p$. In fact we know also that $y$ is the corner of a new NBC which we will obtain by continuing the search and cutting now edges $(x',y')$ where $y'<_h y$. 
Finally we get subtrees $N'_i$ of corner $y_i$ when we put indices at each vertex $y$ found.
 The subtree $N'_0$ is as well the subtree
containing $p$ and is of corner $p$.  The bijection is recursive :
\begin{itemize}
\item Let $L_j^i$ be the LL$k$S corresponding to $N_j^i$  of corner $c_j^i$. We have $h(c)-h({c}_j^i)=i$.
\item Let ${L'}_i$ be the LL$k$S corresponding to ${N'}_i$ of corner ${y}_i$.
\item To every $L_j^i$ such that $i>1$ and $c_j^i<c$ put weight $k-i+1$.
\item To every $L_j^i$ such that $i>1$ and $c_j^i>c$ put weight $k-i$.
\item To $L'_0$ put weight $k-(h(c)-h(p))+1$. Recall that $p$ is the corner of $L'_0$ and that $p<_h c$.
\item To every $L'_i$ put weight $k-(h(c)-h(y_i))$. Now we have that $y_i$ is the corner of $N'_i$ and that $y_i>_h c$.
\item The LL$k$S tree is given by its root $c$ and its left decomposition $D_w$ where $D_w$ is just the set of sub LL$k$S trees to which we have put weight $w$. 
\end{itemize}

\end{proof}

\section{L$k$S and SL$k$S as coloured trees and forests}

We kept the definition of LBS and gave their two similar generalisations because our first goal was to go from NBC to LBS. However we already made the remark that the right decomposition means that LBS are like forests of LLBS. Similarly the L$k$S and SL$k$S are also forests of left parts. We will give in this section a definition which would have been more suitable for our constructions and that is in some sense more natural.

We consider $T{n,k}$ the set of rooted coloured labelled trees that is rooted labelled trees on $[n]$ where the edges have $k$ possible colours in $[k]$. On $[n]$ the number of rooted labelled trees is known to be $n^{n-1}$ and so the cardinality of  $T{n,k}$ is $2^{n-1}n^{n-1}$.

\begin{definition}
An inner vertex $x$ is said to be a \emph{descent } if it has no child of colour 1 or if it has a child $y$ of colour 1 such that $y<x$.

An inner vertex $x$ is said to be a \emph{S-descent } if its smallest child of smallest caller  $y$ is such that $y<x$.
\end{definition}

\begin{theorem}
The LL$k$S are in bijection with the rooted coloured labelled trees such that all inner vertices are descents. The SLL$k$S are in bijection with the rooted coloured labelled trees such that all inner vertices are S-descents.
\end{theorem}

\begin{proof}
From SLL$k$S to labelled trees.
Let $L$ be a LL$k$S, $r$ be its root, $r_i$ be $r$'s number $i$ child 
for $1\le i\le k-1$ (it has no number $k$ child). 
The vertex $r_i$ is the root of an L$k$S $L^i$. Let take the left decomposition of $L^i$
into the set of LL$k$S  $L_0^i$ , \ldots ,$L_{\ell_i}^i$ with corresponding roots
$r_0^i$,\ldots,$r_{\ell_i}^i$. We obtain recursively
the bijection  just by replacing each LL$k$S $L_i$ by its corresponding rooted tree
and adding the edges $\{r,r_j^i\}$ for $0\le j\le\ell_i$ with weight $i$.

From labelled trees to LL$k$S. Let $T$ be a rooted labelled tree with the property and $r$
be its root. Let $T_j^i$ be the rooted labelled trees obtain by deleting the root $r$ where
$i$ is the value of the weight of the deleted edge.
The tree $T$ is not a rooted plane tree so there is special order on these subtrees. 
Any way they have root $r_j^i$ which can be uniquely ordered by their label and
let us suppose then that $r_0^i< r_1^i\ldots <r_{\ell_i}^i$. That means we work on each 
weight separately. By hypothesis we have  $r_0<r$.
So recursively again, we obtain the LL$k$S corresponding to $T$ by replacing
each $T_i$ by its corresponding LL$k$S and joining them to the root by adding the edge
$\{r,r_0^i\}$ ($r_0^i$ is the number $i$ child of $r$) and all the edges $\{r_{j-1}^i,r_j^i\}$ ($r_j^i$ is the number $i$ child of $r_{j-1}^i$) for
$1\le j\le \ell_i$.
\end{proof}

\begin{remark}
\begin{enumerate}
\item Note that LBS are both L$k$S and SL$k$S for $k=1$. Of course this comes from the fact that the two different definitions are the same when $k=1$.
\item 
Note also that we could make a new definition with two different set of weights $W_1=[k_1]$ and $W_2=[k_2]$. A $W_1$-descent would be a inner vertex $x$ such that $y$ its smallest child of smallest weight in $W_1$  is such that $y<x$. Then the descent definition corresponds to the case $k_1=1$ and $k_2=n-1$ and the S-descent definition corresponds to the case $k_1=n$ and $k_2=0$. Other values of $k_1\ge 1$ and $k_2\ge 0$ are possible and will correspond in our construction to the complete gain graphs $K_n^{[1-k_2,k_1+k_2]}$.

\end{enumerate}

\end{remark}


\begin{thebibliography}{9}

\bibitem{A} C.\ A.\ Athanasiadis, Characteristic polynomials of subspace arrangements and finite fields. 
\emph{Adv.\ in Math.}\ 122 (1996), 193--233.
MR 97k:52012.  Zbl 872.52006.

\bibitem{Bjorner} Anders Bj\"orner,
The homology and shellability of matroids and geometric lattices.
In Neil White, ed., \emph{Matroid Applications}, Ch.\ 7, pp.\ 226--283,
Encyc.\ Math.\ Appl., Vol.\ 40,
Cambridge Univ.\ Press, Cambridge, U.K., 1992.
MR 94a:52030.  Zbl.\ 772.05027.

\bibitem{forge} S. Corteel, D. Forge, V. Ventos, Bijections between truncated affine arrangements and valued graphs, to appear in Eur. J. of Comb.,	(arXiv:1403.2573 [math.CO]).

\bibitem{SOA} David Forge and Thomas Zaslavsky,
 Lattice point counts for the Shi arrangement and other affinographic hyperplane arrangements.
\emph{J.\ Combin.\ Theory Ser.\ A} 114(1) (2007), 97--109.
MR 2007i:52026.  Zbl 1105.52014.

\bibitem{GLV} Emeric Gioan and Michel Las Vergnas, The active bijection between regions and simplices in supersolvable arrangements of hyperplanes.  
\emph{Electronic J.\ Combin.}\ (Stanley Festschrift) 11(2) (2006), \#R30, 39pp. 

\bibitem{PS} A.\ Postnikov and R.\ Stanley, Deformations of Coxeter hyperplane arrangements.
\emph{J.\ Combin.\ Theory Ser.\ A} 91 (2000), 544--597.
MR 2002g:52032.  Zbl 962.05004.

\bibitem{St} R.\ P.\ Stanley, \emph{Enumerative Combinatorics}, Vol.\ 1, Wadsworth \& Brooks-Cole,
Belmont, CA, 1986; reprinted by Cambridge Univ.\ Press, Cambridge, U.K., 1997.

\bibitem{BG} Thomas Zaslavsky,
Biased graphs.
I.\ Bias, balance, and gains.
II.\ The three matroids.
III.\ Chromatic and dichromatic invariants.
IV.\ Geometrical realizations.
\emph{J.\ Combin.\ Theory Ser.\ B} {\bf 47} (1989), 32--52;
{\bf 51} (1991), 46--72; {\bf 64} (1995), 17--88; {\bf 89} (2003), 231--297.
 MR 90k:05138; 91m:05056; 96g:05139, 2005b:05057.
 Zbl.\ 714.05057; 763.05096; 950.25778; 1031.05034.

\end{thebibliography}
\end{document}